\documentclass[12pt]{article}
\usepackage[a4paper, margin=1.1in]{geometry}
\usepackage{amsmath,amsthm}
\usepackage{amsfonts,amssymb}
\usepackage{graphicx}
\usepackage{float}
\usepackage{verbatim}
\usepackage{cite}
\usepackage{mathdots}
\newtheorem{thm}{Theorem}[section]
\newtheorem{cor}[thm]{Corollary}

\theoremstyle{remark}
\newtheorem{rem}{\bf{Remark}}[section]

\begin{document}
	\title
	{\bf{Signless Laplacian Energy, Distance Laplacian Energy and Distance Signless Laplacian Spectrum of Unitary Addition Cayley Graphs}}
	\author {\small Naveen Palanivel \footnote{naveenpalanivel.nitc@gmail.com}
		, Chithra.A.V \footnote{chithra@nitc.ac.in} \\ \small Department of Mathematics, National Institute of Technology, Calicut, Kerala, India-673601}
	\date{ }
	\maketitle
\begin{abstract}
	In this paper we compute bounds for signless Laplacian energy, distance signless Laplacian eigenvalues and signless Laplacian energy of unitary addition Cayley graph $G_{n}$. We also obtain distance Laplacian eigenvalues and distance Laplacian energy of $G_{n}$.
\end{abstract}


\section{Introduction}	
Let $G$ be a simple, undirected graph of order $n$ and size $m$ with vertex set $V(G)$ and edge set $E(G)$. The complement $G^{c}$ of a graph $G$ has vertex set $V(G)$ and two vertices are adjacent in $G^{c}$ if and only if they are not adjacent in $G$. The following definitions are found in \cite{abreu2011bounds}, \cite{aouchiche2013two}, \cite{haemers2004enumeration}, \cite{karner2003spectral}, \cite{sinha2011some}, \cite{yang2013bounds}. The transmission $Tr(v)$ of a vertex $v$ is defined to be the sum of the distance from $v$ to all other vertices in $G$, i.e.,
\begin{flalign*}
	Tr(v)=\sum_{u\in V}d(u,v).
\end{flalign*}
A connected graph $G=(V,E)$ is said to be $k-$transmission regular if $Tr(v)=k$ for every vertex $v\in V$.\\
The signless Laplacian matrix of $G$ is the matrix $L^{+}(G)=|L(G)|=Diag(deg)+A(G)$ where $Diag(deg)$ denotes the diagonal matrix of the vertex degree in $G$ and $A(G)$ is the adjacency matrix of $G$. Let $\mu_{1}^{+},\mu_{2}^{+},\cdots,\mu_{n}^{+}$ be the signless Laplacian eigenvalues of $G$. Signless Laplacian energy of $G$ is defined as $LE^{+}(G)=\sum_{i=1}^{n}|\mu_{i}^{+}-\frac{2m}{n}|$.\\
The distance matrix $D=(d_{ij})$ of a connected graph $G$ is the matrix indexed by the vertices of $G$ where $d_{ij}=d(v_{i},v_{j})$. The distance energy $DE(G)$ of a graph $G$ is defined as the sum of the absolute values of the eigenvalues of $D$. The distance Laplacian matrix $D^{L}$ of a connected graph $G$ is the matrix $D^{L}=Diag(Tr)-D$, where $Diag(Tr)$ denotes the diagonal matrix of the vertex transmissions in $G$. Let $\partial_{1}^{L},\partial_{2}^{L},\cdots,\partial_{n}^{L}$ be the distance Laplacian eigenvalues of $G$. The distance signless Laplacian matrix $D^{Q}$ of a connected graph $G$ is the matrix $D^{Q}=Diag(Tr)+D$. Let $\partial_{1}^{Q},\partial_{2}^{Q},\cdots,\partial_{n}^{Q}$ be the distance signless Laplacian eigenvalues of $G$.\\
If $G$ is a $k-$transmission regular graph and the distance spectrum of $G$ are $\lambda_{1}^{D}\geq \lambda_{2}^{D}\geq \cdots \geq \lambda_{n}^{D}$, then\\
1) $k-\lambda_{n}^{D}\geq k-\lambda_{n-1}^{D}\geq \cdots \geq k-\lambda_{1}^{D}$ is the distance Laplacian spectrum of $G$.\\
2) $k+\lambda_{1}^{D}\geq k+\lambda_{2}^{D}\geq \cdots \geq k+\lambda_{n}^{D}$ is the distance signless Laplacian spectrum of $G$.\\
Distance Laplacian energy of $G$ is  $LE_{D}(G)=\sum_{i=1}^{n}|\partial_{i}^{L}-\frac{1}{n}\sum_{j=1}^{n}d_{G}(v_{j})|$, where $d_{G}(v_{j})$ is the sum of distance between $v_{j}$ and the other vertices of $G$.\\
For a positive integer $n>1$ the \textit{unitary Cayley graph} $X_{n}=Cay(Z_{n}, U_{n})$ is the graph whose vertex set is $Z_{n}$, the integers modulo $n$ and if $U_{n}$ denotes set of all units of the ring $Z_{n}$, then two vertices $a,b$ are adjacent if and only if $a-b\in U_{n}$.  \\  
For a positive integer $n>1$, the \textit{unitary addition Cayley graph} $G_{n}$ is the graph whose vertex set is $Z_{n}$, the integers modulo $n$ and if $U_{n}$ denotes set of all units of the ring $Z_{n}$, then two vertices $a,b$ are adjacent if and only if $a+b\in U_{n}$. The unitary addition Cayley graph is \textit{regular} if $n$ is even and \textit{semi regular} if $n$ is odd. \\ 
The right circulant matrix $C_{R}(\bar{c})$ associated to the vector $\bar{c}=(c_{0}, c_{1}, \cdots, c_{n-1})\in R^{n}$ is\\
\begin{align*} 
	C_{R}(\bar{c})=\begin{bmatrix}
		c_{0}& c_{1}& \cdots& c_{n-1}\\
		c_{n-1}& c_{0}& \cdots& c_{n-2}\\
		\vdots& \vdots& \ddots& \vdots\\
		c_{1}& c_{2}& \cdots& c_{0}
	\end{bmatrix}.
\end{align*}
The left circulant matrix $C_{L}(\bar{c})$ associated to the vector $\bar{c}=(c_{0}, c_{1}, \cdots, c_{n-1})\in R^{n}$ is\\
\begin{align*}
	C_{L}(\bar{c})=\begin{bmatrix}
		c_{0}& c_{1}& \cdots& c_{n-1}\\
		c_{1}& c_{2}& \cdots& c_{0}\\
		\vdots& \vdots& \ddots& \vdots\\
		c_{n-1}& c_{0}& \cdots& c_{n-2}
	\end{bmatrix}.
\end{align*}
In \cite{karner2003spectral}, Herbert Karner et al. have shown that $C_{L}(\bar{c})=\Pi C_{R}(\bar{c})$ where $\Pi$ is the orthogonal cyclic shift matrix given by
\begin{align*}
	\Pi=\begin{bmatrix}
		1& 0& \cdots& 0\\
		0& 0& \cdots& 1\\
		\vdots& \vdots& \ddots& \vdots\\
		0& 1& \cdots& 0
	\end{bmatrix}.
\end{align*}
In the same paper, they also proved that the eigenvalues of left circulant matrix $C_{L}(\bar{c})$ are $\lambda_{0}, \pm |\lambda_{1}|,\cdots, \pm |\lambda_{(n-1)/2}|$ if $n$ is odd and $\lambda_{0}, \lambda_{n/2}, \pm |\lambda_{1}|,\cdots, \pm |\lambda_{(n-2)/2}|$ if $n$ is even where $\lambda_{k}$ are the eigenvalues of right circulant matrix $C_{R}(\bar{c})$.\\
The following results found in \cite{ilic2009energy}.\\
Consider the matrix $M=(m_{ij}),0\leq i,j\leq n-1,m_{ij}=\left\{
\begin{array}{l l}
	1 & \quad\text{if gcd(i-j,n)=1},\\
	0 & \quad\text{otherwise}.
\end{array}
\right.$ \\
It is a right circulant matrix, so eigenvalues of $M$ are $\mu(t_{k})\frac{\phi(n)}{\phi(t_{k})}$ for $0\leq k\leq n-1$ where $t_{k}=\frac{n}{gcd(k,n)}$. Here $\phi(n)$ denotes the Euler pi-function and $\mu$ denotes the Mobius function.\\
Again consider the matrix $N=(n_{ij}),0\leq i,j\leq n-1,n_{ij}=\left\{
\begin{array}{l l}
	1 & \quad\text{if gcd(i-j,n)$\neq$ 1},\\
	0 & \quad\text{otherwise}.
\end{array}
\right.$\\
It is again a right circulant matrix, so eigenvalues of $N$ are $n-1-\phi(n),-\mu(t_{k})\frac{\phi(n)}{\phi(t_{k})}-1$ for $1\leq k\leq n-1$.\\
In this paper, we present our work a follows: In section $2$, we recall some results which are necessary for proving the results obtained in this paper. In section $3$, we present bounds for signless Laplacian eigenvalues and signless Laplacian energy of the unitary addition Cayley graph $G_{n}$. In section $4$, we discuss the bounds of signless Laplacian  eigenvalues and signless Laplacian energy of the complement of unitary addition Cayley graph $G_{n}$. In section $5$, we compute bounds for distance eigenvalues and distance energy of the unitary addition Cayley graph $G_{n}$. In Section $6$, we calculate the distance Laplacian eigenvalues and distance Laplacian energy of the unitary addition Cayley graph $G_{n}$. In section $7$, we compute distance signless Laplacian spectrum of the unitary addition Cayley graph $G_{n}$ for $n=p^{m}$ and obtain some bounds of the distance signless Laplacian eigenvalues of the unitary addition Cayley graph $G_{n}$.\\
Throughout this paper $J$ denotes a matrix of order $k$ with all entries are $1$ and $O$ is a null matrix of order $k$.
\section{Preliminaries}
\begin{thm}\textup{\cite{sinha2011some}}\label{isomorphic}
	The unitary addition Cayley graph $G_{n}$ is isomorphic to the unitary Cayley graph $X_{n}$ if and only if $n$ is even.
\end{thm}
\begin{thm}\textup{\cite{key:article}}
	The diameter of the unitary addition Cayley graph $G_{n}$, $n>2$, is 
	\begin{align*}
		diam(G_{n})=\left\{\begin{array}{l l l l}
			2 & \quad \text{if $n$ is prime},\\
			2 & \quad \text{if $n$ is even and $n=2^m$, $m\geq 2$},\\
			3 & \quad \text{if $n$ is even and $n\neq 2^m$,$m \geq 2$},\\
			2 & \quad \text{if $n$ is odd, but not a prime }. \end{array}\right.
	\end{align*}
\end{thm}
\begin{rem}
	If $n$ is even and $n=2^{m}, m\geq 2$, then $Tr(v_{i})=2n-2-\phi(n)$, for every $v_{i} \in V(G_{n})$.
\end{rem}
\begin{rem}
	If $n$ is even with odd prime divisor, then $Tr(v_{i})=\frac{5n}{2}-2-2\phi(n)$, for every $v_{i} \in V(G_{n})$.
\end{rem}
\begin{rem}\label{regular}
	Let $ v_{i} \in V(G_{n}),n$ is odd and $deg(v_{i})=\phi(n)$.Then $Tr(v_{i})=2n-\phi(n)-2$.
\end{rem}
\begin{rem}\label{semiregular}
	Let $ v_{i} \in V(G_{n}),n$ is odd and $deg(v_{i})=\phi(n)-1$. Then $Tr(v_{i})=2n-\phi(n)-1$.
\end{rem}
\begin{thm}\textup{\cite{ilic2010distance}}
	If $n=2^{m}, m\geq 2$, then distance spectrum of the unitary Cayley graph $X_{n}$ is
	\begin{align*}
		\left(\frac{3n}{2}-2, \frac{n}{2}-2, -2, -2,\cdots,-2\right)
	\end{align*}  and distance energy of the unitary Cayley graph $X_{n}$ is $4n-8$.
\end{thm}
\begin{rem}
	If $n=2^{m}, m\geq 2$, then distance Laplacian spectrum of the unitary addition Cayley graph $G_{n}$ is
	\begin{align*}
		\left(\frac{n}{2}-\phi(n), \frac{3n}{2}-\phi(n), 2n-\phi(n), 2n-\phi(n),\cdots, 2n-\phi(n)\right).
	\end{align*}  
\end{rem}
\begin{rem}
	If $n=2^{m}, m\geq 2$, then distance signless Laplacian spectrum of the unitary addition Cayley graph $G_{n}$ is
	\begin{align*}
		\left(2n-\phi(n)-4, \frac{5n}{2}-\phi(n)-4, 2n-\phi(n)-4, 2n-\phi(n)-4,\cdots, 2n-\phi(n)-4\right).
	\end{align*}  
\end{rem}
\begin{thm}\textup{\cite{ilic2010distance}}
	If $n$ is even with odd prime divisor, then distance spectrum of the unitary Cayley graph $X_{n}$ is
	\begin{align*} 
		\left(\frac{5n}{2}-2(\phi(n)+1),2(\phi(n)-1)-\frac{n}{2}, -2-\frac{\mu(t_{k})\phi(n)}{\phi(t_{k})}\right)	  	
	\end{align*}
	where $t_{k}=\frac{n}{gcd(r,n)}, k=1,2,\cdots,\frac{n}{2}-1,\frac{n}{2}+1,\cdots,n-1$ and distance energy of the unitary Cayley graph $X_{n}$ is $\frac{9n-4s-4}{2}+\phi(n)(2^{r+1}-6)+|2\phi(n)-2-\frac{n}{2}|$, where $s=p_{1}p_{2}\cdots p_{r}$ is the maximal square free divisor of n.
\end{thm}
\begin{rem}
	If $n$ is even with odd prime divisor, then distance Laplacian spectrum of the unitary addition Cayley graph $G_{n}$ is
	\begin{align*} 
		\left(0,3n-4\phi(n), \frac{5n}{2}-2\phi(n)+\frac{\mu(t_{k})\phi(n)}{\phi(t_{k})}\right)	 \end{align*}
	where $k=1,2,\cdots,\frac{n}{2}-1,\frac{n}{2}+1,\cdots,n-1$.
\end{rem}
\begin{rem}
	If $n$ is even with odd prime divisor, then distance signless Laplacian spectrum of the unitary addition Cayley graph $G_{n}$ is
	\begin{align*} 
		\left(5n-4\phi(n)-4,2n-4, \frac{5n}{2}-2\phi(n)-\frac{\mu(t_{k})\phi(n)}{\phi(t_{k})}-4\right)	  	
	\end{align*}
	where $k=1,2,\cdots,\frac{n}{2}-1,\frac{n}{2}+1,\cdots,n-1$.
\end{rem}
\begin{thm}\textup{\cite{aouchiche2013two}}
	Let $G$ be a connected graph on $n$ vertices with $diam(G)\le 2$. Let $\mu_{1}\ge \mu_{2}\ge \cdots \ge\mu_{n-1}>\mu_{n}=0$ be the Laplacian spectrum of $G$. Then the distance Laplacian spectrum of $G$ is $2n-\mu_{n-1}\ge 2n-\mu_{n-2}\ge \cdots \ge 2n-\mu_{1}>\partial_{n}^{L}=0$.
\end{thm} 
\begin{thm}\textup{\cite{yang2013bounds}}
	If $G$ is $k-$transmission regular graph, then distance Laplacian energy of $G$ is equal to distance energy of $G$. 
\end{thm} 
\begin{thm}\label{complement energy if n even}\textup{\cite{ilic2009energy}}
	Let $s=p_{1}p_{2}\cdots p_{r}$ be the largest square-free number that divides $n$. Then energy of the complement of unitary Cayley graph $X_{n}$ equals \begin{flalign*}
		E(X_{n}^{c})=2n-2+(2^r-2)\phi(n)-s+\prod_{i=1}^{r}(2-p_{i}).
	\end{flalign*}
\end{thm} 
\begin{thm}\label{eigenvalue bound}\textup{\cite{favaron1993some}}
	Let $A,A_{1},A_{2}$ be three $n\times n$ real symmetric matrices such that $A=A_{1}+A_{2}$. The eigenvalues of these matrices satisfy the following inequalities: for $1\le i\le n$ and $0\le j\le min\{i-1,n-i\},\lambda_{i+j}(A_{1})+\lambda_{n-j}(A_{2})\le \lambda_{i}(A)\le \lambda_{i-j}(A_{1})+\lambda_{1+j}(A_{2})$. 
\end{thm}
\begin{thm}\textup{\cite{naveen2016energy}}
	If $n$ is odd, then Laplacian eigenvalues of the unitary addition Cayley graph $G_{n}$ are $-\mu(t_{k})\frac{\phi(n)}{\phi(t_{k})}+\phi(n)$ for $0\leq k\leq \frac{n-1}{2}$ and $\mu(t_{k})\frac{\phi(n)}{\phi(t_{k})}+\phi(n)$ for $\frac{n+1}{2}\leq k\leq n-1$.
\end{thm}

\section{Bounds for signless Laplacian energy of unitary addition Cayley graphs}
In this section, bounds for signless Laplacian eigenvalues and signless Laplacian energy of the unitary addition Cayley graph $G_{n}$ are computed. 
\begin{thm}
	If $n$ is even, then signless Laplacian eigenvalues of the unitary addition Cayley graph $G_{n}$ are $\mu_{k}^{+}=\phi(n)+\mu(t_{k})\frac{\phi(n)}{\phi(t_{k})}$, where $0\leq k\leq n-1$.
\end{thm}
\begin{cor}
	If $n$ is even, then signless Laplacian energy of the unitary addition Cayley graph $G_{n}$ is $2^r\phi(n)$, where $r$ is the number of distinct prime divisor of $n$.
\end{cor}
\begin{thm}\label{signless Laplacian eigen odd}
	If $n=p^{m},m\geq 1$, then signless Laplacian spectrum of the unitary addition Cayley graph $G_{n}$ is\\
	\begin{align*}
		\begin{pmatrix}
			p^{m}-2p^{m-1}-2& \frac{x_{1}-y_{1}}{2}& p^{m}-p^{m-1}-2& p^{m}-p^{m-1}& p^{m}-2& \frac{x_{1}+y_{1}}{2}\\
			\frac{p-3}{2}& 1& p^{m}-p^{m-1}-p+1& p^{m-1}-1& \frac{p-1}{2}& 1
		\end{pmatrix},
	\end{align*}
	where $x_{1}=(3p^{m}-4p^{m-1}-2)$ and $y_{1}=\sqrt{(p^{m}-2)^{2}+8p^{m-1}}$.
\end{thm}
\begin{proof}
	Let $L^{+}(G_{n})=\begin{bmatrix}
		B & C & \cdots & C\\
		C & B & \cdots & C\\
		\vdots & \vdots & \ddots & \vdots\\
		C & C & \cdots & B
	\end{bmatrix}$ be the signless Laplacian matrix of $G_{n}$ of order $k=p^{m-1}$, \\
	where $B=\begin{bmatrix}
		x & 1 & 1 & \cdots & 1 & 1 & \cdots & 1 & 1\\
		1 & x-1 & 1 & \cdots & 1 & 1 & \cdots & 1 & 0\\
		1 & 1 & x-1 & \cdots & 1 & 1 & \cdots & 0 & 1\\
		\vdots & \vdots & \vdots & \ddots & \vdots & \vdots & \ddots & \vdots & \vdots\\ 
		1 & 1 & 1 & \cdots & x-1 & 0 & \cdots & 1 & 1\\
		1 & 1 & \cdots & 1 & 0 & x-1 & 1 & \cdots & 1 \\
		1 & 1 & \cdots & 0 & 1 & 1 & x-1 & \cdots & 1\\
		\vdots & \vdots & \ddots & \vdots & \vdots & \vdots & \vdots & \ddots & \vdots\\ 
		1 & 0 & \cdots & 1 & 1 & 1 & 1 & \cdots & x-1 
	\end{bmatrix}_{p \times p}$ and $C=C_{L}\begin{bmatrix}
		0 & 1 & 1 & \cdots & 1 & 1
	\end{bmatrix}_{1 \times p}$ and $x=p^{m}-p^{m-1}$.\\
	The matrix $L^{+}(G_{n})$ is permutationally similar to \\
	$\hat{L^{+}}(G_{n})=\begin{bmatrix}
		xI & J & J & \cdots & J & J & \cdots & J & J\\
		J & J+yI & J & \cdots & J & J & \cdots & J & O\\
		J & J & J+yI  & \cdots & J & J & \cdots & O & J\\
		\vdots & \vdots & \vdots & \ddots & \vdots & \vdots & \ddots & \vdots & \vdots\\ 
		J & J & J & \cdots & J+yI  & O & \cdots & J & J\\
		J & J & J & \cdots & O & J+yI  & \cdots & J & J \\
		\vdots & \vdots & \vdots & \iddots & \vdots & \vdots & \ddots & \vdots & \vdots\\
		J & J & O & \cdots & J & J & \cdots  & J+yI & J\\ 
		J & O & J & \cdots & J & J & \cdots & J & J+yI  
	\end{bmatrix}_{p \times p},$ \\
	where $y=p^{m}-p^{m-1}-2$.\\
	Then $L^{+}(G_{n})=P\hat{L^{+}}(G_{n})P^{-1}$ where \\	$P=\begin{bmatrix}
		A_{11}^{T} & A_{12}^{T} & \cdots & A_{1p^{m-2}}^{T} & A_{21}^{T} & A_{22}^{T} & \cdots & A_{2p^{m-2}}^{T} & \cdots & A_{p1}^{T} & A_{p2}^{T} & \cdots & A_{pp^{m-2}}^{T} 
	\end{bmatrix}^{T}$\\ is a permutation matrix of order $p^{m}$, $m\geq 2$, $A_{ij}=(a_{\alpha\beta})$, $a_{\alpha\beta}= \left\{ 
	\begin{array}{l l}
		1 & \quad \text{if $(\alpha, \beta)\in H_{ij}$},\\
		0 & \quad \text{otherwise}.
	\end{array} \right.$ where $\alpha=1,2,\dots,p$, $\beta=1,2,\dots,p^{m}$ and $H_{ij}=\{(1,i+(j-1)p),(2,i+(j-1)p+p^{m-1}),\dots,(p,i+(j-1)p+(p-1)p^{m-1})\}$ where $i=1,2,\dots,p,j=1,2,\dots,p^{m-2}$.\\
	If $m=1$, then $P=I$.
	\begin{flalign*}
		&det(\hat{L^{+}}(G_{n})-\mu^{+} I)\\
		&=\scriptsize{\begin{vmatrix}
				xI-\mu^{+} I & J & J & \cdots & J & J & \cdots & J & J\\
				J & J+yI-\mu^{+} I & J & \cdots & J & J & \cdots & J & O\\
				J & J & J+yI-\mu^{+} I  & \cdots & J & J & \cdots & O & J\\
				\vdots & \vdots & \vdots & \ddots & \vdots & \vdots & \ddots & \vdots & \vdots\\ 
				J & J & J & \cdots & J+yI-\mu^{+} I  & O & \cdots & J & J\\
				J & J & J & \cdots & O & J+yI-\mu^{+} I  & \cdots & J & J \\
				\vdots & \vdots & \vdots & \iddots & \vdots & \vdots & \ddots & \vdots & \vdots\\
				J & J & O & \cdots & J & J & \cdots  & J+yI-\mu^{+} I & J\\ 
				J & O & J & \cdots & J & J & \cdots & J & J+yI-\mu^{+} I  
		\end{vmatrix}}\\
		&=det(J+yI-\mu^{+} I)^{\frac{p-1}{2}}\scriptsize{\begin{vmatrix}
				xI-\mu^{+} I & J-xI+\mu^{+} I & J-xI+\mu^{+} I  & \cdots & J-xI+\mu^{+} I\\
				2J & -J+yI-\mu^{+} I & O & \cdots & O\\
				2J & O & -J+yI-\mu^{+} I  & \cdots & O \\
				\vdots & \vdots & \vdots & \ddots & \vdots \\ 
				2J & O & O & \cdots & -J+yI-\mu^{+} I 
		\end{vmatrix}}\\
		&=det(J+yI-\mu^{+} I)^{\frac{p-1}{2}}det(-J+yI-\mu^{+} I)^{\frac{p-3}{2}}\scriptsize{\begin{vmatrix}
				xI-\mu^{+} I & \left(\frac{p-1}{2}\right)(J-xI+\mu^{+} I)\\
				2J & -J+yI-\mu^{+} I
		\end{vmatrix}}\\
		&=det(J+yI-\mu^{+} I)^{\frac{p-1}{2}}det(-J+yI-\mu^{+} I)^{\frac{p-3}{2}}det[{\mu^{+}}^{2} I+\mu^{+}(2J-xI-yI-pJ)\\
		&-k(p-1)J+x(p-1)J-xJ+xyI]\\
		&=(p^{m}-p^{m-1}-2-\lambda)^{(p^{m}-p^{m-1}-p+1)}(p^{m}-p^{m-1}-\lambda)^{k-1}\\
		& \times (p^{m}-p^{m-1}+k-2-\lambda)^{\frac{p-1}{2}}(p^{m}-p^{m-1}-k-2-\lambda)^{\frac{p-3}{2}}\\
		& \times [{\mu^{+}}^{2}+\mu^{+} (-3p^{m}+4p^{m-1}+2)+(2p^{2m}-6p^{2m-1}+4p^{2m-2}-2p^{m}+2p^{m-1})].&&
	\end{flalign*}
	Thus, the signless Laplacian spectrum of $G_{n}$ is
	\begin{align*}
		\begin{pmatrix}
			p^{m}-2p^{m-1}-2& \frac{x_{1}-y_{1}}{2}& p^{m}-p^{m-1}-2& p^{m}-p^{m-1}& p^{m}-2& \frac{x_{1}+y_{1}}{2}\\
			\frac{p-3}{2}& 1& p^{m}-p^{m-1}-p+1& p^{m-1}-1& \frac{p-1}{2}& 1
		\end{pmatrix},
	\end{align*}
	where $x_{1}=(3p^{m}-4p^{m-1}-2)$ and $y_{1}=\sqrt{(p^{m}-2)^{2}+8p^{m-1}}$.
\end{proof}
\begin{cor}
	If $n=p^{m}, m\geq 2$, then signless Laplacian energy of the unitary addition Cayley graph $G_{n}$ is $2p^{m}-p^{m-1}-2p^{m-2}-p+p^{-1}-2+\sqrt{(p^{m}-2)^{2}+8p^{m-1}}$.
\end{cor}
\begin{cor}
	If $n=p$, then signless Laplacian energy of the unitary addition Cayley graph $G_{n}$ is $p-2-2p^{-1}+\sqrt{(p-2)^{2}+8}$.
\end{cor}
Computation of exact value of signless Laplacian eigenvalues of the unitary addition Cayley graph $G_{n}$ is very difficult for odd values of $n$ except $n=p^{m}$. In such cases we obtain nice bounds for signless Laplacian eigenvalues by  using the definition of signless Laplacian matrix and Theorem \ref{eigenvalue bound}.
\begin{thm}
	If $n$ is odd, then signless Laplacian eigenvalues of the unitary addition Cayley graph $G_{n}$ satisfy the following inequalities:\\ $\mu(t_{k})\frac{\phi(n)}{\phi(t_{k})}+\phi(n)-2\le \mu_{k}^{+}\le  \mu(t_{k})\frac{\phi(n)}{\phi(t_{k})}+\phi(n)$ for $0\le k\le (n-1)/2$ and\\ $-\mu(t_{k})\frac{\phi(n)}{\phi(t_{k})}+\phi(n)-2\le \mu_{k}^{+}\le -\mu(t_{k})\frac{\phi(n)}{\phi(t_{k})}+\phi(n)$ for $(n+1)/2 \le k\le n-1$.
\end{thm}
\begin{proof}
	Let $L^{+}(G_{n})=(l_{ij}^{+})$, $0\leq i, j\leq n-1$, be the signless Laplacian matrix of $G_{n}$, where\\ 
	\[l_{ij}^{+}=\left\{
	\begin{array}{l l l l}
		1 & \quad\text{if gcd(i+j,n)=1 and i$\neq$ j},\\
		\phi(n) & \quad\text{if gcd(i+j,n)$\neq$ 1 and i= j},\\
		\phi(n)-1 & \quad\text{if gcd(i+j,n)=1 and i= j},\\
		0 & \quad\text{otherwise}.
	\end{array}
	\right.\]\\
	Consider $L^{+}(G_{n})=B+C$ where\\
	$B=(b_{ij})$, $0\leq i, j\leq n-1$, $b_{ij}=\left\{
	\begin{array}{l l}
		1 & \quad\text{if gcd(i+j,n)=1},\\
		0 & \quad\text{otherwise},
	\end{array}
	\right.$\\
	and\\
	$C=(c_{ij})$, $0\leq i, j\leq n-1$,  $c_{ij}=\left\{
	\begin{array}{l l l}
		\phi(n) & \quad\text{if gcd(i+j,n)$\neq$ 1 and i= j},\\
		\phi(n)-2 & \quad\text{if gcd(i+j,n)=1 and i= j},\\
		0 & \quad\text{otherwise}.
	\end{array}
	\right.$\\
	From the definition of $B$, we can say $B$ is a left circulant matrix, so eigenvalues of $B$ are $ \mu(t_{k})\frac{\phi(n)}{\phi(t_{k})}$ for $0\le k\le (n-1)/2$ and $-\mu(t_{k})\frac{\phi(n)}{\phi(t_{k})}$ for $(n+1)/2 \le k\le n-1$. Eigenvalues of $C$ are $\phi(n)^{n-\phi(n)}, (\phi(n)-2)^{\phi(n)}$, since $x\in U_{n}\mbox{ implies }2x\in U_{n}$ and $y\in V(G_{n})-U_{n}\mbox{ implies } 2y\in V(G_{n})-U_{n}$.\\
	Thus, the result follows from the eigenvalues of $B,C$ and Theorem \ref{eigenvalue bound}.
\end{proof}
\begin{cor}
	If $n=p_{1}p_{2}\cdots p_{r}$ is odd and square-free number, then signless Laplacian energy of the unitary addition Cayley graph $G_{n}$ satisfy the following inequalities: $\phi(n)\left[2^{r}+1+\frac{1}{n}\right]-n-1-\frac{\phi(n)^{2}}{n}\leq LE^{+}(G_{n})\leq \phi(n)\left[2^{r}+\frac{1}{n}\right]+n-1$.
\end{cor}
\begin{cor}
	If $n$ is odd and non square-free number, then signless Laplacian energy of the unitary addition Cayley graph $G_{n}$ satisfy the following inequalities: $\phi(n)\left[\frac{n(2^{r}+1)-s+1}{n}\right]-s-1\leq LE^{+}(G_{n})\leq \phi(n)\left[\frac{n(2^{r}-1)+s+1}{n}\right]+2n-s-1$ where $s=p_{1}p_{2}\cdots p_{r}$ is the maximal square-free divisor of $n$.
\end{cor}
\section{Bounds for signless Laplacian energy of the complement of unitary addition Cayley graphs}
In this section, we obtain bounds for signless Laplacian eigenvalues and signless Laplacian energy of the complement of unitary addition Cayley graph $G_{n}$.
\begin{thm}
	If $n$ is even, then signless Laplacian eigenvalues of the complement of unitary addition Cayley graph $G_{n}$ are $\phi(n)$ and $n-\phi(n)-\mu(t_{k})\frac{\phi(n)}{\phi(t_{k})}-2$, where $1\leq k\leq n-1$.
\end{thm}
\begin{cor}
	If $n$ is even, then signless Laplacian energy of the complement of unitary addition Cayley graph $G_{n}$ is $2n-2+(2^r-2)\phi(n)-s+\prod_{i=1}^{r}(2-p_{i})$ where $r$ is number of distinct prime divisors.
\end{cor}
\begin{thm}
	If $n=p^{m}, m\geq 1$, then signless Laplacian spectrum of the complement of unitary addition Cayley graph $G_{n}$ is
	\begin{align*}
		\begin{pmatrix}
			0& p^{m-1}-2& p^{m-1}& 2p^{m-1}-2& 2p^{m-1}\\
			\frac{p-1}{2}& p^{m-1}-1& (p-1)(p^{m-1}-1)& 1& \frac{p-1}{2} 
		\end{pmatrix}.
	\end{align*}
\end{thm}
\begin{proof}
	Let $L^{+}(G_{n}^{c})=\begin{bmatrix}
		B & C & \cdots & C\\
		C & B & \cdots & C\\
		\vdots & \vdots & \ddots & \vdots\\
		C & C & \cdots & B
	\end{bmatrix}$ be the signless Laplacian matrix of $G_{n}^{c}$ of order $k=p^{m-1}$, where 
	$B=\begin{bmatrix}
		k-1 & 0 & \cdots & 0 & 0 \\
		0 & k & \cdots & 0 & 1 \\
		0 & 0 & \cdots & 1 & 0 \\
		\vdots & \vdots & \reflectbox{$\ddots$} & \vdots & \vdots \\ 
		0 & 1 & \cdots & 0 & k \\
	\end{bmatrix}_{p \times p}$ \\and $C=C_{L}\begin{bmatrix}
		1 & 0 & \cdots & 0 & 0 
	\end{bmatrix}_{1 \times p}$.\\
	The matrix $L^{+}(G_{n}^{c})$ is permutationally similar to \\
	$\hat{L}^{+}(G_{n}^{c})=\begin{bmatrix}
		J+(k-2)I & O & \cdots & O & O \\
		O & kI & \cdots & O & J \\
		O & O & \cdots & J & O \\
		\vdots & \vdots & \iddots & \vdots & \vdots \\ 
		O & J & \cdots & O & kI \\
	\end{bmatrix}_{p \times p}$.\\
	Then $L^{+}(G_{n}^{c})=P\hat{L}^{+}(G_{n}^{c})P^{-1}$, where $P$ is a matrix given in proof of Theorem \ref{signless Laplacian eigen odd}.
	\begin{flalign*}
		det(\hat{L}^{+}(G_{n}^{c})-\bar{\mu}^{+} I)&=det(J+(k-2)I-\bar{\mu}^{+})det(J+k-\bar{\mu}^{+})^{\frac{p-1}{2}}det(-J+k-\bar{\mu}^{+})^{\frac{p-1}{2}}\\
		&=(2k-2-\bar{\mu}^{+})(k-2-\bar{\mu}^{+})^{k-1}(2k-\bar{\mu}^{+})^{\frac{p-1}{2}}(k-\bar{\mu}^{+})^{\frac{(p-1)(k-1)}{2}}(-\bar{\mu}^{+})^{\frac{p-1}{2}}\\
		&\times(k-\bar{\mu}^{+})^{\frac{(p-1)(k-1)}{2}}\\
		&=(-\bar{\mu}^{+})^{\frac{p-1}{2}}(k-2-\bar{\mu}^{+})^{k-1}(k-\bar{\mu}^{+})^{(p-1)(k-1)}(2k-2-\bar{\mu}^{+})(2k-\bar{\mu}^{+})^{\frac{p-1}{2}}.&&
	\end{flalign*}
	Thus, the signless Laplacian spectrum of $G_{n}^{c}$ is
	\begin{align*}
		\begin{pmatrix}
			0& p^{m-1}-2& p^{m-1}& 2p^{m-1}-2& 2p^{m-1}\\
			\frac{p-1}{2}& p^{m-1}-1& (p-1)(p^{m-1}-1)& 1& \frac{p-1}{2} 
		\end{pmatrix}.
	\end{align*}
\end{proof}
\begin{cor}
	If $n=p^{m}, m\ge 2$, then signless Laplacian energy of the complement of unitary addition Cayley graph $G_{n}$ is $p^{m}+3p^{m-1}-2p^{m-2}-5+3p^{-1}$.
\end{cor}
\begin{cor}
	If $n=p$, then signless Laplacian energy of the complement of unitary addition Cayley graph $G_{n}$ is $p-p^{-1}$.
\end{cor}
\begin{thm}
	If $n$ is odd, then signless Laplacian eigenvalues of the complement of unitary addition Cayley graph $G_{n}$, $\bar{\mu}_{k}^{+}$, satisfy the following inequalities:\\  $\mu(t_{k})\frac{\phi(n)}{\phi(t_{k})}+n-\phi(n)-2\le \bar{\mu}_{k}^{+}\le \mu(t_{k})\frac{\phi(n)}{\phi(t_{k})}+n-\phi(n)$ for $1\le k\le (n-1)/2$ and $-\mu(t_{k})\frac{\phi(n)}{\phi(t_{k})}+n-\phi(n)-2\le \bar{\mu}_{k}^{+}\le -\mu(t_{k})\frac{\phi(n)}{\phi(t_{k})}+n-\phi(n)$ for $(n+1)/2 \le k\le n-1$ and $2n-2\phi(n)-2\le \bar{\mu}_{0}^{+}\le 2n-2\phi(n)$.
\end{thm}
\begin{proof}
	Let $L^{+}(G_{n}^{c})=({l_{ij}^{+}}^{c})$, $0\leq i, j\leq n-1$, be the signless Laplacian matrix of $G_{n}^{c}$, where\\ 
	\[{l_{ij}^{+}}^{c}=\left\{
	\begin{array}{l l l l}
		1 & \quad\text{if gcd(i+j,n)$\neq$ 1 and i$\neq$ j},\\
		n-1-\phi(n) & \quad\text{if gcd(i+j,n)$\neq$ 1 and i= j},\\
		n-\phi(n) & \quad\text{if gcd(i+j,n)= 1 and i= j},\\
		0 & \quad\text{otherwise}.
	\end{array}
	\right.\]\\
	Consider $L^{+}(G_{n}^{c})=E+F$ where\\
	$E=(e_{ij})$, $0\leq i, j\leq n-1$, $e_{ij}=\left\{
	\begin{array}{l l}
		1 & \quad\text{if gcd(i+j,n)$\neq$ 1},\\
		0 & \quad\text{otherwise},
	\end{array}
	\right.$\\
	and\\
	$F=(f_{ij})$, $0\leq i, j\leq n-1$,  $f_{ij}=\left\{
	\begin{array}{l l l}
		n-2-\phi(n) & \quad\text{if gcd(i+j,n)$\neq$ 1 and i= j},\\
		n-\phi(n) & \quad\text{if gcd(i+j,n)=1 and i= j},\\
		0 & \quad\text{otherwise}.
	\end{array}
	\right.$\\
	We have $E$ is a left circulant matrix, so eigenvalues of $E$ are $ n-\phi(n), \mu(t_{k})\frac{\phi(n)}{\phi(t_{k})}$ for $1\le k\le (n-1)/2$ and $-\mu(t_{k})\frac{\phi(n)}{\phi(t_{k})}$ for $(n+1)/2 \le k\le n-1$.\\         	
	Eigenvalues of $F$ are $[n-2-\phi(n)]^{n-\phi(n)}, [n-\phi(n)]^{\phi(n)}$, since $x\in U_{n}\mbox{ implies }2x\in U_{n}$ and $y\in V(G_{n})-U_{n}\mbox{ implies } 2y\in V(G_{n})-U_{n}$.\\
	The result is obtained from the eigenvalues of $E,F$ and Theorem \ref{eigenvalue bound}.
\end{proof}
\begin{cor}
	If $n=p_{1}p_{2}\cdots p_{r}$ is odd and square-free number, then signless Laplacian energy of the complement of unitary addition Cayley graph $G_{n}$ satisfy the following inequalities: $\phi(n)\left[2^{r}-2-\frac{1}{n}\right]+\frac{\phi(n)^{2}}{n}\leq LE^{+}(G_{n}^{c})\leq \phi(n)\left[2^{r}-2-\frac{1}{n}\right]+2n$.
\end{cor}
\begin{cor}
	If $n$ is odd and non square-free number, then signless Laplacian energy of the complement of unitary addition Cayley graph $G_{n}$ satisfy the following inequalities: $\phi(n)\left[\frac{n(2^{r}-3)+s-1}{n}\right]+2n-2s\leq LE^{+}(G_{n}^{c})\leq \phi(n)\left[\frac{n(2^{r}-1)-s-1}{n}\right]+2n$ where $s=p_{1}p_{2}\cdots p_{r}$ is the maximal square-free divisor of $n$.
\end{cor}
\section{Bounds for distance eigenvalues and energy of unitary addition Cayley graphs}
Here, we compute distance energy of the unitary addition Cayley graph $G_{n}$ for $n=p^m,m\geq 1$ and bounds of the distance energy of the unitary addition Cayley graph $G_{n}$ for $n$ is odd.

\begin{thm}
	If $n=p^{m}, m\geq 1$, then distance spectrum of the unitary addition Cayley graph $G_{n}$ is
	\begin{align*}
		\begin{pmatrix}
			-1-p^{m-1}& -2 & -1& \frac{x_{2}-y_{2}}{2}& p^{m-1}-1& \frac{x_{2}+y_{2}}{2}\\
			\frac{p-1}{2}& p^{m-1}-1& (p-1)(p^{m-1}-1)& 1& \frac{p-3}{2}& 1
		\end{pmatrix},
	\end{align*}
	where $x_{2}=p^{m}+2p^{m-1}-3$ and $y_{2}=\sqrt{p^{2m}+4p^{2m-1}-12p^{2m-2}+2p^{m}-4p^{m-1}+1}$.
\end{thm}
\begin{proof}
	Let $D(G_{n})=\begin{bmatrix}
		B & C & \cdots & C\\
		C & B & \cdots & C\\
		\colon & \colon & \ddots & \colon\\
		C & C & \cdots & B
	\end{bmatrix}$ be the distance matrix of $G_{n}$ of order $k=p^{m-1}$, where  $B=\begin{bmatrix}
		0 & 1 & 1 & \cdots & 1 & 1 & \cdots & 1 & 1\\
		1 & 0 & 1 & \cdots & 1 & 1 & \cdots & 1 & 2\\
		1 & 1 & 0 & \cdots & 1 & 1 & \cdots & 2 & 1\\
		\colon & \colon & \colon & \ddots & \colon & \colon & \ddots & \colon & \colon\\ 
		1 & 1 & 1 & \cdots & 0 & 2 & \cdots & 1 & 1\\
		1 & 1 & \cdots & 1 & 2 & 0 & 1 & \cdots & 1 \\
		1 & 1 & \cdots & 2 & 1 & 1 & 0 & \cdots & 1\\
		\colon & \colon & \ddots & \colon & \colon & \colon & \colon & \ddots & \colon\\ 
		1 & 2 & \cdots & 1 & 1 & 1 & 1 & \cdots & 0 
	\end{bmatrix}_{p \times p}$ \\and 
	$C=C_{L}\begin{bmatrix}
		2 & 1 & 1 & \cdots & 1 & 1
	\end{bmatrix}_{1 \times p}$.\\
	The matrix $A$ is permutationally similar to \\
	$\hat{D}(G_{n})=\begin{bmatrix}
		2J-2I & J & J & \cdots & J & J & \cdots & J & J\\
		J & J-I & J & \cdots & J & J & \cdots & J & 2J\\
		J & J & J-I  & \cdots & J & J & \cdots & 2J & J\\
		\vdots & \vdots & \vdots & \ddots & \vdots & \vdots & \ddots & \vdots & \vdots\\ 
		J & J & J & \cdots & J-I  & 2J & \cdots & J & J\\
		J & J & J & \cdots & 2J & J-I  & \cdots & J & J \\
		\vdots & \vdots & \vdots & \iddots & \vdots & \vdots & \ddots & \vdots & \vdots\\
		J & J & 2J & \cdots & J & J & \cdots  & J-I & J\\ 
		J & 2J & J & \cdots & J & J & \cdots & J & J-I  
	\end{bmatrix}_{p \times p}$.\\
	Then $A=P\hat{D}(G_{n})P^{-1}$, where $P$ is a matrix given in proof of Theorem \ref{signless Laplacian eigen odd}.
	\begin{flalign*}
		&det(\hat{D}(G_{n})-\lambda^{D} I)\\ 
		&=\scriptsize{\begin{vmatrix}
				2J-2I-\lambda^{D} I & J & J & \cdots & J & J & \cdots & J & J\\
				J & J-\lambda^{D} I & J & \cdots & J & J & \cdots & J & 2J\\
				J & J & J-\lambda^{D} I  & \cdots & J & J & \cdots & 2J & J\\
				\vdots & \vdots & \vdots & \ddots & \vdots & \vdots & \ddots & \vdots & \vdots\\ 
				J & J & J & \cdots & J-\lambda^{D} I  & 2J & \cdots & J & J\\
				J & J & J & \cdots & 2J & J-\lambda^{D} I  & \cdots & J & J \\
				\vdots & \vdots & \vdots & \iddots & \vdots & \vdots & \ddots & \vdots & \vdots\\
				J & J & 2J & \cdots & J & J & \cdots  & J-\lambda^{D} I & J\\ 
				J & 2J & J & \cdots & J & J & \cdots & J & J-\lambda^{D} I  
		\end{vmatrix}}\\
		&=det(-J-\lambda^{D} I)^{\frac{p-1}{2}}\scriptsize{\begin{vmatrix}
				2J-2I-\lambda^{D} I & -J+2I+\lambda^{D} I & -J+2I+\lambda^{D} I  & \cdots & -J+2I+\lambda^{D} I\\
				2J & J-\lambda^{D} I & O & \cdots & O\\
				2J & O & J-\lambda^{D} I  & \cdots & O \\
				\vdots & \vdots & \vdots & \ddots & \vdots \\ 
				2J & O & O & \cdots & J-\lambda^{D} I 
		\end{vmatrix}}\\
		&=det(-J-\lambda^{D} I)^{\frac{p-1}{2}}det(J-\lambda^{D} I)^{\frac{p-3}{2}}\scriptsize{\begin{vmatrix}
				2J-2I-\lambda^{D} I & \left(\frac{p-1}{2}\right)(-J+2I+\lambda^{D} I)\\
				2J & J-\lambda^{D} I
		\end{vmatrix}}\\
		&=det(-J-\lambda^{D} I)^{\frac{p-1}{2}}det(J-\lambda^{D} I)^{\frac{p-3}{2}}det[\lambda^{2D} I+\lambda^{D}(-2J+2-pJ)\\
		&+k(p+1)J-2pJ]\\
		&=(-2-\lambda^{D})^{k-1}(-1-\lambda^{D})^{(p-1)(k-1)}(k+1+\lambda^{D})^{\frac{p-1}{2}}(-k+1+\lambda^{D})^{\frac{p-3}{2}}[-\lambda^{2D}\\
		&+\lambda^{D}(2k+kp-3)-k^{2}p-k^{2}+2kp+2k-2].&&
	\end{flalign*}
	Thus, the distance spectrum of $G_{n}$ is
	\begin{align*}
		\begin{pmatrix}
			-1-p^{m-1}& -2 & -1& \frac{x_{2}-y_{2}}{2}& p^{m-1}-1& \frac{x_{2}+y_{2}}{2}\\
			\frac{p-1}{2}& p^{m-1}-1& (p-1)(p^{m-1}-1)& 1& \frac{p-3}{2}& 1
		\end{pmatrix},
	\end{align*}
	where $x_{2}=p^{m}+2p^{m-1}-3$ and $y_{2}=\sqrt{p^{2m}+2p^{m}-4p^{m-1}+1}$.
\end{proof}
\begin{cor}
	If $n=p^{m}$ then distance energy of the unitary addition Cayley graph $G_{n}$ is	$3p^{m}+p^{m-1}-p-3$.
\end{cor}
When $n$ is odd, we need the following results to compute the bound for eigenvalues of the unitary addition Cayley graph $G_{n}$:\\

Let $n$ be odd. Then eigenvalues of $D(X_{n})$ are $\lambda_{0}^{D}=2(n-1)-\phi(n)$ and $\lambda_{r}^{D}=-2-\mu(t_{r})\frac{\phi(n)}{\phi(t_{r})}$, $1\le r\le n-1$\textup{\cite{ilic2010distance}}.\\
So the eigenvalues of $D(X_{n})+2I$ are $\lambda_{0}^{D}=2n-\phi(n)$ and $\lambda_{r}^{D} =-\mu(t_{r})\frac{\phi(n)}{\phi(t_{r})}$, $1\le r\le n-1$.\\
Consider $B=(b_{ij})$, $0\le i, j\le n-1$, where \[b_{ij}=\left\{
\begin{array}{l l}
	1 & \quad\text{if gcd(i+j, n)=1},\\
	2 & \quad\text{otherwise}. 
\end{array} \right.\]\\ 
It is a left circulant matrix with first row $(c_{0}, c_{1}, \cdots, c_{n-1})$ where \\
\[ c_{j}= \left\{ 
\begin{array}{l l}
	1 & \quad \text{if gcd(j, n)=1},\\
	2 & \quad \text{otherwise}.
\end{array} \right.\] \\
If $n$ is odd, then eigenvalues of $B$  are $2n-\phi(n),\mu(t_{k})\frac{\phi(n)}{\phi(t_{k})}$ for $1\le k\le (n-1)/2$ and $-\mu(t_{k})\frac{\phi(n)}{\phi(t_{k})}$ for $(n+1)/2\le k\le n-1$.
\begin{thm}\label{distance eigenvalue}
	If $n$ is odd, then distance eigenvalues of unitary addition Cayley graph $G_{n}$ satisfy the following inequalities:\\ $\mu(t_{k})\frac{\phi(n)}{\phi(t_{k})}-2\leq \lambda_{k}^{D}\leq\mu(t_{k})\frac{\phi(n)}{\phi(t_{k})}-1$ for $1\le k\le (n-1)/2$ and $-\mu(t_{k})\frac{\phi(n)}{\phi(t_{k})}-2\leq \lambda_{k}^{D}\leq -\mu(t_{k})\frac{\phi(n)}{\phi(t_{k})}-1$ for $(n+1)/2\le k\le n-1$ and $2n-\phi(n)-2\leq \lambda_{0}^{D}\leq 2n-\phi(n)-1$.
\end{thm}
\begin{proof}
	Let $D=(d_{ij})$, $0\leq i, j\leq n-1$ be the distance matrix of unitary addition Cayley graph $G_{n}$, where\\ 
	\[d_{ij}=\left\{
	\begin{array}{l l}
		1 & \quad\text{if gcd(i+j,n)=1 and i$\ne$ j},\\
		2 & \quad\text{otherwise}.
	\end{array}
	\right.\]\\  
	Then $D=B+C$ where\\
	$B=(b_{ij})$, $0\leq i, j\leq n-1$, $b_{ij}=\left\{
	\begin{array}{l l}
		1 & \quad\text{if gcd(i+j,n)= 1},\\
		2 & \quad\text{otherwise},
	\end{array}
	\right.$\\
	and\\
	$C=(c_{ij})$, $0\leq i, j\leq n-1$, $c_{ij}=\left\{
	\begin{array}{l l l}
		-1 & \quad\text{if gcd(i+j,n)=1 and i=j},\\
		-2 & \quad\text{if gcd(i+j,n)$\neq$1 and i=j},\\
		0 & \quad\text{otherwise}.
	\end{array}
	\right.$\\ 	
	From the definition of $B$, we can say $B$ is a left circulant matrix, so eigenvalues of $B$ are $2n-\phi(n),\mu(t_{k})\frac{\phi(n)}{\phi(t_{k})}$ for $1\le k\le (n-1)/2$ and $-\mu(t_{k})\frac{\phi(n)}{\phi(t_{k})}$ for $(n+1)/2\le k\le n-1$.\\   	
	Eigenvalues of $C$ are $-1^{\phi(n)}, -2^{n-\phi(n)}$, since $x\in U_{n}\mbox{ implies }2x\in U_{n}$ and $y\in V(G_{n})-U_{n}\mbox{ implies } 2y\in V(G_{n})-U_{n}$.\\
	Thus, the result follows from the eigenvalues of $B,C$ and Theorem \ref{eigenvalue bound}.
\end{proof}
\begin{cor}
	If $n$ is odd, then distance energy of the unitary addition Cayley graph $G_{n}$ satisfy the following inequalities: $(2^{r}-2)\phi(n)+\frac{4n-s-3}{2}\leq DE(G_{n})\leq (2^{r}-2)\phi(n)+\frac{6n-s-3}{2}$ where $s=p_{1}p_{2}\cdots p_{r}$ is the maximal square-free divisor of $n$.
\end{cor}
\section{Distance Laplacian energy of unitary addition Cayley graphs}  
In this section, we present distance Laplacian eigenvalues and distance Laplacian energy of the unitary addition Cayley graph $G_{n}$.
\begin{thm}
	If $n$ is odd, then distance Laplacian eigenvalues of the unitary addition Cayley graph $G_{n}$ are $2n +\mu(t_{k})\frac{\phi(n)}{\phi(t_{k})}-\phi(n)$ for $1\leq k\leq \frac{n-1}{2}$, $2n-\mu(t_{k})\frac{\phi(n)}{\phi(t_{k})}-\phi(n)$ for $\frac{n+1}{2}\leq k\leq n-1$ and $0$.
\end{thm}   
\begin{thm}\label{distance Laplacian energy}
	If $n$ is odd and non square-free number, then distance Laplacian energy of $G_{n}$ is 
	\begin{flalign*}
		LE_{D}(G_{n})=(n-s)\left[2-\frac{\phi(n)}{n}\right]+(2^{r}-1)\phi(n)+2n-2-\left(\frac{n-1}{n}\right)\phi(n).
	\end{flalign*} 
\end{thm}
\begin{proof}
	Distance Laplacian energy of $G_{n}$ is,
	\begin{flalign}
		LE_{D}(G_{n})&=\sum_{i=1}^{n-1}\left|2n-\mu_{i}-\frac{1}{n}\sum_{j=1}^{n}d_{G}(v_{j})\right|+2n-2-\left(\frac{n-1}{n}\right)\phi(n) \nonumber\\
		&=\sum_{i=1}^{n-1}\left|2n-\mu_{i}-2n+2+\phi(n)-\frac{\phi(n)}{n}\right|+2n-2-\left(\frac{n-1}{n}\right)\phi(n)\nonumber\\
		&=\sum_{1\le i\le \frac{n-1}{2}}\left|\frac{\mu\left(\frac{n}{gcd(i,n)}\right)\phi(n)}{\phi\left(\frac{n}{gcd(i,n)}\right)}+2-\frac{\phi(n)}{n}\right|+\sum_{\frac{n+1}{2}\le i\le n-1}\left|-\frac{\mu\left(\frac{n}{gcd(i,n)}\right)\phi(n)}{\phi\left(\frac{n}{gcd(i,n)}\right)}+2-\frac{\phi(n)}{n}\right|\nonumber\\
		&+2n-2-\left(\frac{n-1}{n}\right)\phi(n).\label{1}&&
	\end{flalign}
	Divide the sum in equation \eqref{1} into two parts,
	\begin{enumerate}
		\item $\frac{n}{gcd(i,n)}$ is a non square-free number 
		\item $\frac{n}{gcd(i,n)}$ is a square-free number(SF) 
	\end{enumerate}
	Then $LE_{D}(G_{n})=S_{1}+S_{2}+S_{3}+2n-2-\left(\frac{n-1}{n}\right)\phi(n)$, where
	\begin{flalign}
		&S_{1}=\sum_{1\le i\le n-1,\frac{n}{gcd(i,n)}\notin SF}\left|2-\frac{\phi(n)}{n}\right|,\nonumber\\
		&S_{2}=\sum_{1\le i\le \frac{n-1}{2},\frac{n}{gcd(i,n)}\in SF}\left|\frac{\mu\left(\frac{n}{gcd(i,n)}\right)\phi(n)}{\phi\left(\frac{n}{gcd(i,n)}\right)}+2-\frac{\phi(n)}{n}\right| \mbox{and }\nonumber\\
		&S_{3}=\sum_{\frac{n+1}{2}\le i\le n-1,\frac{n}{gcd(i,n)}\in SF}\left|-\frac{\mu\left(\frac{n}{gcd(i,n)}\right)\phi(n)}{\phi\left(\frac{n}{gcd(i,n)}\right)}+2-\frac{\phi(n)}{n}\right|.\label{2}&&
	\end{flalign}
	Part I: Suppose $\frac{n}{gcd(i,n)}$ is a non square-free number.\\
	We know that number of solutions of the equation $\mu\left(\frac{n}{gcd(i,n)}\right)=0$ is $n-s$.\\
	Therefore 
	\begin{flalign}
		S_{1}=(n-s)\left[2-\frac{\phi(n)}{n}\right].\label{3}
	\end{flalign}
	Part II: Suppose $\frac{n}{gcd(i,n)}$ is a square-free number.\\
	In $S_{2}$ and $S_{3}$ two possibility arises one is $\frac{n}{gcd(i,n)}$ has an even number of distinct prime divisors and the other one is $\frac{n}{gcd(i,n)}$ has an odd number of distinct prime divisors. We can denote the corresponding sum in $S_{2}$ by $S_{4}$ and $S_{5}$ and $S_{3}$ by $S_{6}$ and $S_{7}$. Now we have two sub cases.\\
	Sub case I:
	\begin{flalign}
		\mbox{Assume }S_{2}&=\sum_{S_{4}\cup S_{5}}\left|\frac{\mu\left(\frac{n}{gcd(i,n)}\right)\phi(n)}{\phi\left(\frac{n}{gcd(i,n)}\right)}+2-\frac{\phi(n)}{n}\right|\nonumber\\  	  &=\sum_{S_{4}}\left|\frac{\mu\left(\frac{n}{gcd(i,n)}\right)\phi(n)}{\phi\left(\frac{n}{gcd(i,n)}\right)}+2-\frac{\phi(n)}{n}\right|+\sum_{S_{5}}\left|\frac{\mu\left(\frac{n}{gcd(i,n)}\right)\phi(n)}{\phi\left(\frac{n}{gcd(i,n)}\right)}+2-\frac{\phi(n)}{n}\right|\nonumber\\
		&=\sum_{S_{4}}\left|\frac{\phi(n)}{\phi\left(\frac{n}{gcd(i,n)}\right)}+2-\frac{\phi(n)}{n}\right|+\sum_{S_{5}}\left|-\frac{\phi(n)}{\phi\left(\frac{n}{gcd(i,n)}\right)}+2-\frac{\phi(n)}{n}\right|\nonumber\\
		&=\sum_{S_{4}}\left[\frac{\phi(n)}{\phi\left(\frac{n}{gcd(i,n)}\right)}+2-\frac{\phi(n)}{n}\right]+\sum_{S_{5}}\left[\frac{\phi(n)}{\phi\left(\frac{n}{gcd(i,n)}\right)}-2+\frac{\phi(n)}{n}\right]\nonumber\\
		&=\phi(n)\sum_{S_{4}}\frac{1}{\phi\left(\frac{n}{gcd(i,n)}\right)}+2\sum_{S_{4}}1-\frac{\phi(n)}{n}\sum_{S_{4}}1+\phi(n)\sum_{S_{5}}\frac{1}{\phi\left(\frac{n}{gcd(i,n)}\right)}\nonumber\\
		&-2\sum_{S_{5}}1+\frac{\phi(n)}{n}\sum_{S_{5}}1\nonumber\\
		&=\phi(n)\sum_{S_{4}\cup S_{5}}\frac{1}{\phi\left(\frac{n}{gcd(i,n)}\right)}+2\left[\sum_{S_{4}}1-\sum_{S_{5}}1\right]-\frac{\phi(n)}{n}\left[\sum_{S_{4}}1-\sum_{S_{5}}1\right]\label{4}.&&
	\end{flalign}
	Sub case II:
	\begin{flalign}
		\mbox{Suppose }S_{3}&=\sum_{S_{6}\cup S_{7}}\left|-\frac{\mu\left(\frac{n}{gcd(i,n)}\right)\phi(n)}{\phi\left(\frac{n}{gcd(i,n)}\right)}+2-\frac{\phi(n)}{n}\right|\nonumber\\
		&=\phi(n)\sum_{S_{6}\cup S_{7}}\frac{1}{\phi\left(\frac{n}{gcd(i,n)}\right)}-2\left[\sum_{S_{6}}1-\sum_{S_{7}}1\right]+\frac{\phi(n)}{n}\left[\sum_{S_{6}}1-\sum_{S_{7}}1\right]\label{5}.&&
	\end{flalign}
	From \eqref{2}, \eqref{3}, \eqref{4} and \eqref{5},
	\begin{flalign*}
		LE_{D}(G_{n})&=(n-s)\left[2-\frac{\phi(n)}{n}\right]+\phi(n)\left(2^{r}-1\right)+2\left[\sum_{S_{4}}1-\sum_{S_{5}}1-\sum_{S_{6}}1+\sum_{S_{7}}1\right]\nonumber\\
		&-\frac{\phi(n)}{n}\left[\sum_{S_{4}}1-\sum_{S_{5}}1-\sum_{S_{6}}1+\sum_{S_{7}}1\right]+2n-2-\left(\frac{n-1}{n}\right)\phi(n)\nonumber\\     	
		&=(n-s)\left[2-\frac{\phi(n)}{n}\right]+\phi(n)\left(2^{r}-1\right)+2n-2-\left(\frac{n-1}{n}\right)\phi(n).&&
	\end{flalign*}      	
	Thus, distance Laplacian energy of the unitary addition Cayley graph $LE_{D}(G_{n})$ is
	\begin{flalign*}
		(n-s)\left[2-\frac{\phi(n)}{n}\right]+(2^{r}-1)\phi(n)+2n-2-\left(\frac{n-1}{n}\right)\phi(n).
	\end{flalign*} 
\end{proof}
\begin{thm}
	If $n$ is odd and $n=p_{1}p_{2}\cdots p_{r}$, where $p_{1}<p_{2}<\cdots <p_{r}$, then distance Laplacian energy of $G_{n}$ is 
	\begin{flalign*}
		LE_{D}(G_{n})=\phi(n)\left[2-\frac{\phi(n)}{n}\right]+(2^{r}-2)\phi(n)+2n-2-\left(\frac{n-1}{n}\right)\phi(n).
	\end{flalign*}
\end{thm}
\begin{proof}
	Proof of this theorem is similar to Theorem \ref{distance Laplacian energy}.
\end{proof}
\section{Bounds for distance signless Laplacian eigenvalues of unitary addition Cayley graphs}
Here, we compute distance signless Laplacian spectrum of the unitary addition Cayley graph $G_{n}$ and bounds of the distance signless Laplacian eigenvalues of the unitary addition Cayley graph $G_{n}$.

\begin{thm}
	If $n=p^{m}, m\geq 1$, then distance signless 
	Laplacian spectrum of the unitary addition Cayley graph $G_{n}$ is
	\begin{align*}
		\begin{pmatrix}
			p^{m}-2& p^{m}+p^{m-1}-4 & p^{m}+p^{m-1}-2& \frac{x_{3}-y_{3}}{2}& p^{m}+2p^{m-1}-2& \frac{x_{3}+y_{3}}{2}\\
			\frac{p-1}{2}& p^{m-1}-1& (p-1)(p^{m-1}-1)& 1& \frac{p-3}{2}& 1
		\end{pmatrix},
	\end{align*}
	where $x_{3}=3p^{m}+4p^{m-1}-6$ and $y_{3}=\sqrt{p^{2m}+4p^{m}-8p^{m-1}+4}$.
\end{thm}
\begin{proof}
	Let $D^{Q}(G_{n})=\begin{bmatrix}
		B & C & \cdots & C\\
		C & B & \cdots & C\\
		\vdots & \vdots & \ddots & \vdots\\
		C & C & \cdots & B
	\end{bmatrix}$ be the distance signless Laplacian matrix of $G_{n}$ of order $k=p^{m-1}$, where  $B=\begin{bmatrix}
		l-1 & 1 & 1 & \cdots & 1 & 1 & \cdots & 1 & 1\\
		1 & l & 1 & \cdots & 1 & 1 & \cdots & 1 & 2\\
		1 & 1 & l & \cdots & 1 & 1 & \cdots & 2 & 1\\
		\vdots & \vdots & \vdots & \ddots & \vdots & \vdots & \ddots & \vdots & \vdots\\ 
		1 & 1 & 1 & \cdots & l & 2 & \cdots & 1 & 1\\
		1 & 1 & \cdots & 1 & 2 & l & 1 & \cdots & 1 \\
		1 & 1 & \cdots & 2 & 1 & 1 & l & \cdots & 1\\
		\vdots & \vdots & \ddots & \vdots & \vdots & \vdots & \vdots & \ddots & \vdots\\ 
		1 & 2 & \cdots & 1 & 1 & 1 & 1 & \cdots & l 
	\end{bmatrix}_{p \times p}$ and
	$C=C_{L}\begin{bmatrix}
		2 & 1 & 1 & \cdots & 1 & 1
	\end{bmatrix}_{1 \times p}$ and $l=2n-\phi(n)-1$.\\
	The matrix $D^{Q}(G_{n})$ is permutationally similar to \\
	$\hat{D}^{Q}(G_{n})=\begin{bmatrix}
		2J+\hat{x}I & J & J & \cdots & J & J & \cdots & J & J\\
		J & J+\hat{y}I & J & \cdots & J & J & \cdots & J & 2J\\
		J & J & J+\hat{y}I  & \cdots & J & J & \cdots & 2J & J\\
		\vdots & \vdots & \vdots & \ddots & \vdots & \vdots & \ddots & \vdots & \vdots\\ 
		J & J & J & \cdots & J+\hat{y}I  & 2J & \cdots & J & J\\
		J & J & J & \cdots & 2J & J+\hat{y}I  & \cdots & J & J \\
		\vdots & \vdots & \vdots & \iddots & \vdots & \vdots & \ddots & \vdots & \vdots\\
		J & J & 2J & \cdots & J & J & \cdots  & J+\hat{y}I & J\\ 
		J & 2J & J & \cdots & J & J & \cdots & J & J+\hat{y}I  
	\end{bmatrix}_{p \times p},$ \\
	where $\hat{x}=p^{m}+p^{m-1}-4$ and $\hat{y}=p^{m}+p^{m-1}-2$.\\
	Then $D^{Q}(G_{n})=P\hat{D}^{Q}(G_{n})P^{-1}$, where $P$ is a matrix given in the proof of Theorem \ref{signless Laplacian eigen odd}.
	\begin{flalign*}
		& det(\hat{D}^{Q}(G_{n})-\partial^{Q} I)\\
		&=\scriptsize{\begin{vmatrix}
				2J+\hat{x}I-\partial^{Q} I & J & J & \cdots & J & J & \cdots & J & J\\
				J & J+\hat{y}I-\partial^{Q} I & J & \cdots & J & J & \cdots & J & O\\
				J & J & J+\hat{y}I-\partial^{Q} I  & \cdots & J & J & \cdots & 2J & J\\
				\vdots & \vdots & \vdots & \ddots & \vdots & \vdots & \ddots & \vdots & \vdots\\ 
				J & J & J & \cdots & J+\hat{y}I-\partial^{Q} I  & 2J & \cdots & J & J\\
				J & J & J & \cdots & 2J & J+\hat{y}I-\partial^{Q} I  & \cdots & J & J \\
				\vdots & \vdots & \vdots & \iddots & \vdots & \vdots & \ddots & \vdots & \vdots\\
				J & J & 2J & \cdots & J & J & \cdots  & J+\hat{y}I-\partial^{Q} I & J\\ 
				J & 2J & J & \cdots & J & J & \cdots & J & J+\hat{y}I-\partial^{Q} I  
		\end{vmatrix}}\\
		&=det(-J+\hat{y}I-\partial^{Q} I)^{\frac{p-1}{2}}\scriptsize{\begin{vmatrix}
				2J+\hat{x}I-\partial^{Q} I & -J-\hat{x}I+\partial^{Q} I & -J-\hat{x}I+\partial^{Q} I  & \cdots & -J-\hat{x}I+\partial^{Q} I\\
				2J & J+\hat{y}I-\partial^{Q} I & O & \cdots & O\\
				2J & O & J+\hat{y}I-\partial^{Q} I  & \cdots & O \\
				\vdots & \vdots & \vdots & \ddots & \vdots \\ 
				2J & O & O & \cdots & J+\hat{y}I-\partial^{Q} I 
		\end{vmatrix}}\\
		&=det(-J+\hat{y}I-\partial^{Q} I)^{\frac{p-1}{2}}det(J+\hat{y}I-\partial^{Q} I)^{\frac{p-3}{2}}\scriptsize{\begin{vmatrix}
				2J+\hat{x}I-\partial^{Q} I & \left(\frac{p-1}{2}\right)(-J-\hat{x}I+\partial^{Q} I)\\
				2J & J+\hat{y}I-\partial^{Q} I
		\end{vmatrix}}\\
		&=det(-J+\hat{y}I-\partial^{Q} I)^{\frac{p-1}{2}}det(J+\hat{y}I-\partial^{Q} I)^{\frac{p-3}{2}}det[{\partial^{Q}}^{2} I+\partial^{Q}(-2J-\hat{x}I-\hat{y}I-pJ)\\
		&+k(p-1)J+\hat{x}(p-1)J+\hat{x}J+2\hat{y}J+\hat{x}\hat{y}I+2kJ]\\
		&=(p^{m}+p^{m-1}-4-\partial^{Q})^{k-1}(p^{m}+p^{m-1}-2-\partial^{Q})^{(p-1)(k-1)}(p^{m}-2-\partial^{Q})^{\frac{p-1}{2}}\nonumber\\
		&\times (p^{m}+2p^{m-1}-2-\partial^{Q})^{\frac{p-3}{2}}[{\partial^{Q}}^{2}-\partial^{Q}(3p^{m}+4p^{m-1}-6)\nonumber\\
		&+(2p^{2m}+6p^{2m-1}+4p^{2m-2}-10p^{m}-10p^{m-1}+8)].&&
	\end{flalign*}
	Thus, the distance signless Laplacian spectrum of $G_{n}$ is
	\begin{align*}
		\begin{pmatrix}
			p^{m}-2& p^{m}+p^{m-1}-4 & p^{m}+p^{m-1}-2& \frac{x_{3}-y_{3}}{2}& p^{m}+2p^{m-1}-2& \frac{x_{3}+y_{3}}{2}\\
			\frac{p-1}{2}& p^{m-1}-1& (p-1)(p^{m-1}-1)& 1& \frac{p-3}{2}& 1
		\end{pmatrix},
	\end{align*}
	where $x_{3}=3p^{m}+4p^{m-1}-6$ and $y_{3}=\sqrt{p^{2m}+4p^{m}-8p^{m-1}+4}$.
\end{proof}
\begin{thm}
	If $n$ is odd, then distance signless Laplacian eigenvalues of the unitary addition Cayley graph $G_{n}$ satisfy the following inequalities:\\ $\mu(t_{k})\frac{\phi(n)}{\phi(t_{k})}+2n-\phi(n)-4\le \partial_{k}^{Q}\le \mu(t_{k})\frac{\phi(n)}{\phi(t_{k})}+2n-\phi(n)-2$ for $1\le k\le (n-1)/2$ and $-\mu(t_{k})\frac{\phi(n)}{\phi(t_{k})}+2n-\phi(n)-4\le \partial_{k}^{Q}\le -\mu(t_{k})\frac{\phi(n)}{\phi(t_{k})}+2n-\phi(n)-2$ for  $(n+1)/2 \le k\le n-1$ and $4n-2\phi(n)-4\le \partial_{0}^{Q}\le 4n-2\phi(n)-2$.
\end{thm}
\begin{proof}
	This result is an immediate consequence of Remarks \ref{regular}, \ref{semiregular}, Theorems \ref{eigenvalue bound}, \ref{distance eigenvalue} and definition of distance signless Laplacian matrix of $G_{n}$.
\end{proof}
\section{Conclusion}
In this paper we obtain the following relation between the different types of eigenvalues of the unitary addition Cayley graph $G_{n}$ for all $n$.\\
1) $2\lambda_{k}^{D}<\mu_{k}^{+}<\partial_{k}^{Q}<\partial_{k}^{L}, 1\leq k\leq n-1$ and $\partial_{0}^{L}<\mu_{0}^{+}<\lambda_{0}^{D}<\partial_{k}^{Q}$.

\bibliography{ref4}
\bibliographystyle{plain}

\end{document}